\documentclass[12pt,twoside,reqno]{amsart}

\usepackage[OT1]{fontenc}    %
\usepackage{type1cm}         %
\usepackage[english]{babel}  %
\usepackage{amsthm}          %
\usepackage[dvips]{graphicx} 
\usepackage[bf]{caption2}    
\usepackage{psfrag}          

\numberwithin{equation}{section}

\theoremstyle{plain}
\newtheorem{theorem}{Theorem}[section]
\newtheorem{corollary}[theorem]{Corollary}
\newtheorem{proposition}[theorem]{Proposition}
\newtheorem{lemma}[theorem]{Lemma}

\theoremstyle{remark}

\newtheorem{example}[theorem]{Example}
\newtheorem*{ack}{Acknowledgement}     

\setcaptionwidth{0.9\textwidth}

\newcounter{counter}
\newenvironment{labeledlist}[1]
{
  \renewcommand{\thecounter}{({#1}\arabic{counter})}
  \begin{list}{({#1}\arabic{counter})}{\usecounter{counter}}
}
{
  \end{list}
}

\newcommand{\R}{\mathbb{R}}

\newcommand{\N}{\mathbb{N}}

\newcommand{\HH}{\mathcal{H}}

\newcommand{\hhh}{\mathtt{h}}
\newcommand{\iii}{\mathtt{i}}
\newcommand{\jjj}{\mathtt{j}}

\newcommand{\roo}{\varrho}
\newcommand{\fii}{\varphi}
\newcommand{\ksi}{\xi}
\newcommand{\as}{\underline{s}}
\newcommand{\ys}{\overline{s}}

\DeclareMathOperator{\dimh}{dim_H}

\DeclareMathOperator{\dimt}{dim_T}
\DeclareMathOperator{\dist}{dist}
\DeclareMathOperator{\diam}{diam}
\DeclareMathOperator{\inter}{int}

\begin{document}

\title{Geometric rigidity of a class of fractal sets}

\author{Antti K\"aenm\"aki} 

\address{Department of Mathematics and Statistics \\
         P.O. Box 35 (MaD) \\
         FIN-40014 University of Jyv\"askyl\"a \\
         Finland}
\email{antakae@maths.jyu.fi}

\date{\today}

\begin{abstract}
  We study geometric rigidity of a class of fractals, which is
  slightly larger than the collection of self-conformal sets. Namely,
  using a new method, we shall prove that a set of this class is
  contained in a smooth submanifold or is totally spread out.
\end{abstract}

\maketitle

\section{Introduction}

We study limit sets of certain iterated function systems on $\R^d$. A
self-conformal set is a limit set of an iterated function system in
which the mappings are conformal on a neighborhood of the limit
set. To define the class of limit sets we are interested in, we use
mappings that are required to be conformal only on the limit set. With
the conformality here, we mean that the derivative of the mapping is
an orthogonal transformation. This class is larger than the collection
of self-conformal sets.

To illustrate the type of results we are interested in, we recall
the following known theorems dealing with self-conformal sets.
The latter one is a generalization of Mattila's rigidity
theorem for self-similar sets (\cite[Corollary 4.3]{PM2}).
The method we use in this paper delivers a new proof and
generalization of these
theorems. To find other rigidity results of similar kind,
the reader is referred to \cite{MMU} and \cite{OS}. Let $E$ be a
self-conformal set, $\HH^t$ denote the $t$-dimensional Hausdorff
measure, and $\dimt$ and $\dimh$ be the topological dimension and the
Hausdorff dimension, respectively.

\begin{theorem}[Mayer and Urba\'nski \mbox{\cite[Corollary 1.3]{VMU}}]
  \label{thm:mayerurbanski}
  Suppose $l = \dimt(E)$. Then either

  (1) $\dimh(E) > l$ or

  (2) $E$ is contained in an $l$-dimensional
  affine subspace or an $l$-dimensional geometric sphere whenever $d$
  exceeds $2$ and if $d$ equals $2$, $E$ is contained in an analytic
  curve.
\end{theorem}

\begin{theorem}[K\"aenm\"aki \mbox{\cite[Theorem 2.1]{AK1}}] \label{thm:kaenmaki}
  Suppose $t = \dimh(E)$ and $0<l<d$. Then either

  (1) $\HH^t(E \cap M)=0$ for every $l$-dimensional
  $C^1$-submanifold $M \subset \R^d$ or

  (2) $E$ is contained in an $l$-dimensional
  affine subspace or an $l$-dimensional geometric sphere whenever $d$
  exceeds $2$ and if $d$ equals $2$, $E$ is contained in an analytic
  curve.
\end{theorem}

Our aim is to prove results of similar kind for the previously
mentioned class of limit sets. We define the class rigorously in the
next chapter.

\section{Class of fractal sets}

We consider the sets obtained as geometric projections of the symbol
space $I^\infty$: Take a finite set $I$ with at least two elements and
set $I^* = \bigcup_{n=1}^\infty I^n$ and $I^\infty = I^\N$. If
$\iii \in I^*$ and $\jjj \in I^* \cup I^\infty$, then with the
notation $\iii,\jjj$ we mean the element obtained by juxtaposing the
terms of $\iii$ and $\jjj$. The \emph{length} of $\iii$, that is, the
number of terms in $\iii$, is denoted by $|\iii|$. Let $X \subset
\R^d$ be a compact set and choose a collection $\{ X_\iii : \iii \in
I^* \}$ of nonempty closed subsets of $X$ satisfying
\begin{labeledlist}{L}
  \item $X_{\iii,i} \subset X_\iii$ for every $\iii
  \in I^*$ and $i \in I$, \label{L1}
  \item $\diam(X_\iii) \to 0$ as $|\iii| \to \infty$. \label{L2}
\end{labeledlist}
Now the \emph{projection mapping} is the function $\pi \colon I^\infty
\to X$ for which
\begin{equation*}
  \{ \pi(\iii) \} = \bigcap_{n=1}^\infty X_{\iii|_n}
\end{equation*}
when $\iii \in I^\infty$. The compact set $E = \pi(I^\infty)$ is
called a \emph{limit set}.

Since this setting is too general to study the geometry, we assume the
limit set is constructed by using sets of the form $X_\iii =
\fii_\iii(X)$, where $\fii_\iii = \fii_{i_1} \circ \cdots \circ
\fii_{i_{|\iii|}}$
for $\iii = (i_1,\ldots,i_{|\iii|}) \in I^*$ and the mappings $\fii_i$
belong into the following category: Suppose $\Omega' \subset \R^d$ is open and
$\Omega$ is open and bounded such that $\overline{\Omega} \subset
\Omega'$ and $X \subset \Omega$. We consider mappings $\fii \in
C^2(\Omega')$ for which $\fii(X) \subset X$ and
\begin{labeledlist}{F}
  \item there exist constants $0 < \as,\ys < 1$ for which $\ys^2 \le
  \as$ and
    \begin{equation*}
      \as \le |(\fii'(x))^{-1}|^{-1} \le
      |\fii'(x)| \le \ys
    \end{equation*}
    when $x \in \Omega$, \label{F1}
  \item the derivative of $\fii$ is an orthogonal transformation on
  $E$, that is,
    \begin{equation*}
      |(\fii'(x))^{-1}|^{-1} = |\fii'(x)|
    \end{equation*}
    when $x \in E$. \label{F3}
\end{labeledlist}
Here $| \,\cdot\, |$ denotes the usual operator norm for linear
mappings. Furthermore, we set $|| \fii_\iii' || = \sup_{x \in
  \Omega}|\fii_\iii'(x)|$.

For example, each contractive conformal mapping satisfies both
assumptions \ref{F1} and \ref{F3}. At first glance, it might seem that
requiring mappings that define the limit set to be conformal on the
limit set, to be a very restrictive assumption for nonconformal
mappings. In the following, we shall give an example of a nonconformal
setting.

\begin{example} \label{example}
  Suppose the mappings $\fii_1,\ldots,\fii_k$ defined on an open
  set $\Omega' \subset \R^d$ are conformal (see \cite[page 22]{YR} for
  definition) and contractive on an open
  and bounded set $\Omega$ for which $\overline{\Omega} \subset
  \Omega'$. Assume also that there is a compact set $X \subset \Omega$
  such that $\fii_i(X) \subset X$ for each $i \in \{ 1,\ldots,k
  \}$. The limit set $E$ associated to this setting is called a
  \emph{self-conformal set}. Furthermore, we require that $\max_i
  ||\fii_i'||^2 \max_i ||(\fii_i^{-1})'|| < 1$.

  Next choose $\max_i ||\fii_i'|| < \ys < 1$ and $0 < \as < (\max_i
  ||(\fii_i^{-1})'||)^{-1}$ such that $\ys^2 < \as$. Suppose 
  $h \colon \R^d \to \R^d$ is a $C^2$ diffeomorphism such that
  it is conformal on $E$. We assume also that
  \begin{equation} \label{eq:h_oletus}
    1 \le ||h'|| \, ||(h^{-1})'|| \le \min\left\{ \frac{\ys}{\max_i
    ||\fii_i'||}, \frac{1}{\as\max_i ||(\fii_i^{-1})'||} \right\}.
  \end{equation}

  Define $\tilde\fii_i = h \circ \fii_i \circ h^{-1}$ for every $i \in
  \{ 1,\ldots,k \}$ and set $\tilde\Omega' = h(\Omega')$,
  $\tilde\Omega = h(\Omega)$, and $\tilde X = h(X)$. Since
  $\tilde\fii_i(\tilde X) \subset \tilde X$ for every $i$, the assumption
  \ref{L1} is satisfied for the collection $\{ \tilde\fii_\iii(\tilde
  X)$ : $\iii \in I^*$ \}. We claim that also the assumption \ref{L2} is
  satisfied and the mappings $\tilde\fii_i$ satisfy the assumptions
  \ref{F1} and \ref{F3}. To see this, notice that
  \begin{equation} \label{eq:tilde_omin}
  \begin{split}
    |\tilde\fii_i'(x)| &\le |(h'(h^{-1}(x)))^{-1}| |h'(\fii_i \circ
    h^{-1}(x))| |\fii_i'(h^{-1}(x))|, \\
    |(\tilde\fii_i'(x))^{-1}| &\le |h'(h^{-1}(x))| |(h'(\fii_i \circ
    h^{-1}(x)))^{-1}| |(\fii_i'(h^{-1}(x)))^{-1}|
  \end{split}
  \end{equation}
  for every $x \in \tilde\Omega$. The condition \ref{F1}, and hence also
  the condition \ref{L2}, can now be verified by using
  \eqref{eq:h_oletus}. Denoting the limit set associated to this
  setting with $\tilde E$, it is straightforward to see that $\tilde E
  = h(E)$. Assumptions on $h$ guarantee that the equations in
  \eqref{eq:tilde_omin} hold with equality provided that $x \in
  \tilde E$. Therefore also \ref{F3} holds.

  \begin{figure}[!t]
  \psfrag{h}{$h$}
  \begin{center}
  \includegraphics[scale=0.6]{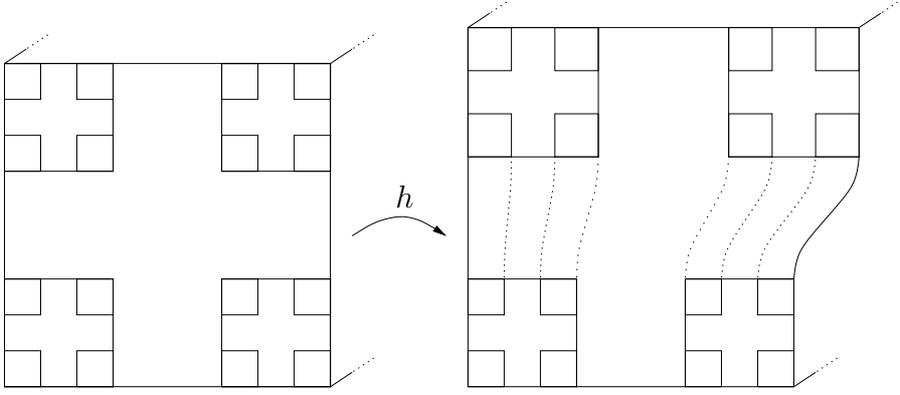}
  \end{center}
  \caption{A nonconformal example.}
  \label{fig:nonconformal}
  \end{figure}

  The class of limit sets obtained by this method clearly includes all
  the self-conformal sets. Since the collection of mappings that
  generate the limit set is not necessarily unique,
  we shall
  next give an example of a self-conformal set $E$ and a mapping $h$
  such that there are no conformal mappings having $h(E)$ as the limit
  set. Let $E$ be the usual Cantor dust on $\R^3$, that is, $E = C^3$,
  where $C$ is the middle third Cantor set on the unit interval. 
  Define $h \colon \R^3 \to \R^3$ such that $h(x,y,z) = g(z)(x,y,z)$,
  where $g$ is an increasing $C^2$ function with the following
  properties: $g' <  c_1$, $g \equiv 1$ on $[0,\tfrac13]$ and $g
  \equiv c_2$ on $[\tfrac23,1]$, see Figure
  \ref{fig:nonconformal}. Now, with suitable choices of
  $0<\as<\tfrac13$, $\tfrac13<\ys<1$, $c_1>0$, and $c_2>1$, the
  mapping $h$ satisfies the condition \eqref{eq:h_oletus}.
  If the set $h(E)$ were a limit set of a collection of conformal
  mappings, it would be invariant with respect to these
  mappings. Hence there exists a conformal mapping taking a cylinder
  set small enough (if $\Omega$ is connected, then a first level
  cylinder would suffice) to the whole set $h(E)$ such that the image
  of a $2$-dimensional affine subspace containing one side of the small
  cylinder set includes sides of two first level cylinder sets located
  in two distinct $2$-dimensional affine subspaces (the sides on the
  right in Figure \ref{fig:nonconformal}). According to Liouville's
  Theorem (for example, see \cite[Theorem 4.1]{YR}) this is not possible.
  Therefore, the class of limit sets obtained by this method is strictly
  larger than the collection of all self-conformal sets.
\end{example}


To avoid too much overlapping among
the sets $\fii_\iii(X)$, we assume the open set condition, that is,
$\fii_i\bigl( \inter(X) \bigr) \cap \fii_j\bigl( \inter(X)
\bigr) = \emptyset$ for $i \ne j$, and the existence of $\roo_0 > 0$
for which
\begin{equation} \label{eq:bc}
  \inf_{x \in \partial X} \inf_{0<r<\roo_0} \frac{\HH^d\bigl( B(x,r)
  \cap \inter(X) \bigr)}{\HH^d\bigl( B(x,r) \bigr)} > 0,
\end{equation}
where $\partial X$ denotes the boundary of $X$. These assumptions are
crucial in determining the conformal measure, see \eqref{eq:3.x}.
From now on, without mentioning it explicitly, this is the setting we
are working with.

As a consequence of the assumption \ref{F1}, we have the following
proposition. Observe that the assumption \ref{F3} is not needed here.

\begin{proposition}[Falconer \mbox{\cite[Proposition 4.3]{KF}}]
  \label{thm:falconer}
  There exists a constant $c > 0$ such that
  \begin{equation*}
    |\fii_\iii'(x) - \fii_\iii'(y)| \le c|\fii_\iii'(x)||x-y|
  \end{equation*}
  for every $\iii \in I^*$ and $x,y \in \Omega$.
\end{proposition}

As a corollary, Falconer \cite[Corollary 4.4]{KF} shows that there exists
a bounded function $1 \le K(t) \le K_0$, $K(t) \to 1$ as $t
\to 0$, such that
\begin{equation} \label{eq:BDP}
\begin{split}
  |\fii_\iii'(x)| &\le K(|x-y|)\, |\fii_\iii'(y)|, \\
  |(\fii_\iii'(x))^{-1}|^{-1} &\le K(|x-y|)\, |(\fii_\iii'(y))^{-1}|^{-1}
\end{split}
\end{equation}
for every $\iii \in I^*$ and $x,y \in \Omega$.
In the following, $B(a,r)$ denotes the open ball centered at $a \in
\R^d$ with radius $r>0$. The closed ball is denoted by
$\overline{B}(a,r)$ whereas the closure of a given set $A$ is denoted
with $\overline{A}$. The boundary of $A$ is denoted by $\partial
A$. Finally, we set $[x,y] = \{ \lambda x + (1-\lambda)y : 0 \le
\lambda \le 1 \}$.

\begin{lemma} \label{thm:bdpomin}
  (1) If $x \in E$, then
  \begin{equation*}
    B\bigl( \fii_\iii(x),K_0^{-1}|\fii_\iii'(x)|r \bigr) \subset
    \fii_\iii\bigl( B(x,r) \bigr)
  \end{equation*}
  for every $\iii \in I^*$ and $0 < r < \dist(E,\partial\Omega)$.

  (2) If $x \in X$, then
  \begin{equation*}
    \fii_\iii\bigl( B(x,r) \bigr) \subset B\bigl(
    \fii_\iii(x),||\fii_\iii'||r \bigr)
  \end{equation*}
  for every $\iii \in I^*$ and $0 < r < \dist(X,\partial\Omega)$.

  (3) There exists a constant $D \ge 1$ such that
  \begin{equation*}
    \diam\bigl( \fii_\iii(X) \bigr) \le D||\fii_\iii'||
  \end{equation*}
  for every $\iii \in I^*$.
\end{lemma}

\begin{proof}
  We shall prove (1). The proofs of (2) and (3) are rather routine and
  will be omitted.
  Take $x \in E$, $\iii \in I^*$, and $0 < r <
  \dist(E,\partial\Omega)$. Iterating \ref{F3} and using
  \eqref{eq:BDP}, we have
  \begin{equation} \label{eq:multibdp}
    |\fii_\iii'(x)| \le K(|x-y|)\, |(\fii_\iii'(y))^{-1}|^{-1}
  \end{equation}
  when $y \in \Omega$. Let $r_1 > 0$ be the supremum of all radii for
  which $B\bigl( \fii_\iii(x),r_1 \bigr) \subset \fii_\iii\bigl( B(x,r)
  \bigr)$. Using now the Mean Value Theorem, we find, for each $z,w \in
  \overline{B}\bigl( \fii_\iii(x),r_1 \bigr)$ and $\theta \in \R^d$, a
  point $\ksi \in [z,w]$ such that
  \begin{equation*}
    \theta\cdot\bigl( \fii_\iii^{-1}(z) - \fii_\iii^{-1}(w) \bigr) =
    \theta\cdot\bigl( (\fii_\iii^{-1})'(\ksi)(z-w) \bigr).
  \end{equation*}
  Thus, choosing $\theta = (x - y) / |x - y|$, where $y \in \partial
  B(x,r)$ is such that $\fii_\iii(y) \in \partial B\bigl(
  \fii_\iii(x),r_1 \bigr)$, we get, using \eqref{eq:multibdp},
  \begin{equation} \label{eq:laskelma}
  \begin{split}
    r &= |x-y| = \bigl| \fii_\iii^{-1}\bigl( \fii_\iii(x) \bigr) -
    \fii_\iii^{-1}\bigl( \fii_\iii(y) \bigr) \bigr| \\
    &\le |(\fii_\iii^{-1})'(\ksi)||\fii_\iii(x) - \fii_\iii(y)| \\
    &= \bigl| \bigl( \fii_\iii'(\fii_\iii^{-1}(\ksi)) \bigr)^{-1}
    \bigr| |\fii_\iii(x) - \fii_\iii(y)| \\
    &\le K(|\fii_\iii^{-1}(\ksi) - x|)\, |\fii_\iii'(x)|^{-1}
    |\fii_\iii(x) - \fii_\iii(y)|,
  \end{split}
  \end{equation}
  where $\ksi \in [\fii_\iii(x),\fii_\iii(y)]$. Hence
  $K_0^{-1}|\fii_\iii'(x)|r \le r_1$, which finishes the proof.
%
\end{proof}

\section{Geometric rigidity}

We shall first set down some notation.
Let $0<l<d$ be an integer and $G(d,l)$ the collection of all
$l$-dimensional linear subspaces of $\R^d$. The orthogonal projection onto
$V \in G(d,l)$ is denoted by $P_V$. We denote the orthogonal complement
of $V$ with $V^\perp \in G(d,d-l)$ and the projection onto that
by $Q_V = P_{V^\perp}$. We can metricize $G(d,l)$ by identifying
$V \in G(d,l)$ with the projection $Q_V$ and defining for $V,W \in G(d,l)$
\begin{equation*}
  d(V,W) = |Q_V - Q_W|,
\end{equation*}
where $| \,\cdot\, |$ is the usual operator norm for linear mappings.
With this metric, $G(d,l)$ is compact.
Furthermore, we denote $V + \{ x \} = \{ v+x : v \in V \}$ for $x \in \R^d$
and $AV = \{ Av : v \in V \}$ for a nonsingular linear mapping
$A \colon \R^d \to \R^d$.

If $a \in \R^d$, $V\in G(d,l)$, $0<\delta < 1$, and $r>0$, we set
\begin{align*}
  X(a,V,\delta) &= \{ x \in \R^d : |Q_V(x-a)| < \delta^{1/2}|x-a| \}, \\
  X(a,r,V,\delta) &= X(a,V,\delta) \cap B(a,r), \\
  V_a(\delta) &= \{ x \in \R^d : |Q_V(x-a)| < \delta \}.
\end{align*}
Notice that the closure of $X(a,V,\delta)$ is the complement of
$X(a,V^{\perp},1-\delta)$. Salli \cite{AS} has shown that $d(V,W) =
\sup_{x \in V \cap S^{d-1}} \dist(x,W)$. Hence the set $X(0,V,\delta)$
is an open ball in $G(d,l)$ centered at $V$ with radius $\delta^{1/2}$.

For the purpose of verifying our main result, we need the following
lemma. In the lemma we study images of small angles.
We work in the setting described in the previous chapter.

\begin{lemma} \label{thm:smallangles}
  Suppose $a \in E$, $\iii \in I^*$, $0 < l < d$, $0<\delta<1$,
  $\tfrac{1}{2} \le \roo < 1$, and $V \in G(d,l)$. Then there exists
  $r_0 = r_0(\delta,\roo) > 0$ depending only on $\delta$ and $\roo$
  such that
  \begin{equation*}
    \fii_\iii\bigl( X(a,r,V,\roo\delta) \bigr) \subset X\bigl(
    \fii_\iii(a),||\fii_\iii'||r,\fii_\iii'(a)V,\delta \bigr)
  \end{equation*}
  whenever $0< r <r_0$.
\end{lemma}

\begin{proof}
  First of all, choose $r_0 > 0$ small enough such that $r_0 <
  \dist(E,\partial\Omega)$. Then by Lemma \ref{thm:bdpomin}(2) we have
  $\fii_\iii\bigl( B(a,r) \bigr) \subset B\bigl(
  \fii_\iii(a),||\fii_\iii'||r \bigr) \subset \Omega$ for every $0 < r
  < r_0$. Take $0 < r < r_0$ and $x \in X(a,r,V,\roo\delta)$. Denote
  $V' = \fii_\iii'(a)V$, $y = P_V(x-a) + a$, and $\theta =
  Q_{V'}\bigl( \fii_\iii(x) - \fii_\iii(a) \bigr) / \bigl|
  Q_{V'}\bigl( \fii_\iii(x) - \fii_\iii(a) \bigr) \bigr|$. Using the
  Mean Value Theorem, we choose $\ksi \in [x,a]$ such that
  \begin{equation} \label{eq:val}
  \begin{split}
    \bigl|Q_{V'}\bigl( \fii_\iii(x) - \fii_\iii(a) \bigr)\bigr|  &=
    \theta \cdot \bigl( \fii_\iii(x) - \fii_\iii(a) \bigr) \\
    &= \theta \cdot \bigl( \fii_\iii'(\ksi)(x-a) \bigr).
  \end{split}
  \end{equation}
  Since $\fii_\iii'(a)(y-a) \in V'$, we have
  \begin{equation} \label{eq:sa_lasku1}
  \begin{split}
    \bigl| Q_{V'}\bigl( \fii_\iii(x) - \fii_\iii(a) \bigr) \bigr| &=
    \bigl| \theta \cdot \bigl( \fii_\iii(x) - \fii_\iii(a) -
    \fii_\iii'(a)(x-a) \\
    &\qquad - \fii_\iii'(a)(y-a) + \fii_\iii'(a)(x-a)
    \bigr) \bigr| \\
    &\le \bigl| \theta \cdot \bigl( \fii_\iii(x) - \fii_\iii(a) -
    \fii_\iii'(a)(x-a) \bigr) \bigr| \\
    &\qquad + \bigl| \theta \cdot \bigl(
    \fii_\iii'(a)(y-a) - \fii_\iii'(a)(x-a) \bigr) \bigr| \\
    &\le |\fii_\iii'(\ksi)(x-a) - \fii_\iii'(a)(x-a)| +
    |\fii_\iii'(a)(x-y)|
  \end{split}
  \end{equation}
  using \eqref{eq:val} and the Cauchy-Schwartz inequality. Calculating
  as in \eqref{eq:laskelma}, we notice that
  \begin{equation} \label{eq:laskelma2}
    |\fii_\iii'(a)||x-a| \le K(|\fii_\iii^{-1}(\ksi') -
     a|)\, |\fii_\iii(x) - \fii_\iii(a)|,
  \end{equation}
  where $\ksi' \in [\fii_\iii(x),\fii_\iii(a)]$. Observe
  that $|\fii_\iii^{-1}(\ksi') - a| \le
  K_0|\fii_\iii'(a)|^{-1}|\fii_\iii(x) - \fii_\iii(a)| \le K_0^2|x-a|$
  by \eqref{eq:BDP}. Therefore, when $|x-a|$ is small, also
  $|\fii_\iii^{-1}(\ksi') - a|$ is small, and hence, to simplify the
  notation, we may replace in the following 
  $K(|\fii_\iii^{-1}(\ksi') - a|)$ with $K(|x-a|)$.
  Using Proposition \ref{thm:falconer} and \eqref{eq:laskelma2}, we
  obtain
  \begin{equation} \label{eq:sa_lasku2}
  \begin{split}
    |\fii_\iii'(\ksi)(x-a) - \fii_\iii'&(a)(x-a)| \le
    |\fii_\iii'(\ksi) - \fii_\iii'(a)||x-a| \\
    &\le c|\fii_\iii'(a)||\ksi - a||x-a| \\
    &\le cK(|x-a|)\, |\fii_\iii(x) - \fii_\iii(a)||x-a|.
  \end{split}
  \end{equation}
  Using \eqref{eq:laskelma2}, we also have
  \begin{equation} \label{eq:sa_lasku3}
    \frac{|\fii_\iii'(a)(x-y)|}{|\fii_\iii(x) - \fii_\iii(a)|} \le
    K(|x-a|)\frac{|\fii_\iii'(a)||x-y|}{|\fii_\iii'(a)||x-a|} \le
    K(|x-a|)(\roo\delta)^{1/2}
  \end{equation}
  and hence, combining \eqref{eq:sa_lasku1}, \eqref{eq:sa_lasku2}, and
  \eqref{eq:sa_lasku3}, we conclude
  \begin{equation*}
    \frac{\bigl| Q_{V'}\bigl( \fii_\iii(x) - \fii_\iii(a) \bigr)
    \bigr|}{|\fii_\iii(x) - \fii_\iii(a)|} \le K(|x-a|)\bigl(c|x-a| +
    (\roo\delta)^{1/2}\bigr).
  \end{equation*}
  Finally, choosing $r_0 \le
  \delta^{1/2}c^{-1}\bigl(((\roo+1)/2)^{1/2}-\roo^{1/2}\,\bigr)$
  so small such that $K(t) \le (2/(\roo+1))^{1/2}$ for all
  $0 < t \le r_0$, we have finished the proof.
\end{proof}

With this geometrical lemma we are able to study tangents of the limit set
$E$. Let $m$ be a Borel measure
on $E$, $0<l<d$, and $t > 0$. Take $a \in E$ and $V \in G(d,l)$.
We say that $V$ is a \emph{weak $(t,l)$-tangent plane for $E$ at $a$} if
\begin{equation*}
  \liminf_{r \downarrow 0} \frac{m\bigl( B(a,r) \setminus V_a(\delta r)
    \bigr)}{r^t} = 0
\end{equation*}
for all $0<\delta<1$. Observe that this concept does not depend on $m$
if there exists a constant $C > 0$ such that $m\bigl( B(x,r) \bigr)
\ge Cr^t$ for all $x \in E$ and $0<r<r_0$. We also say that $V$ is an
\emph{$l$-tangent plane for $E$ at $a$} if for every $0<\delta<1$
there exists $r > 0$ such that
\begin{equation} \label{eq:inkluusio}
  E \cap B(a,r) \subset X(a,V,\delta).
\end{equation}
Furthermore, the set $E$ is said to be
\emph{uniformly $l$-tangential} if for each $0<\delta<1$ there exists
$r>0$ such that for every point $a \in E$ there is $V \in G(d,l)$ such
that \eqref{eq:inkluusio} holds.
An application of Whitney's Extension Theorem
shows that a uniformly $l$-tangential set is a subset of an
$l$-dimensional $C^1$-submanifold, see Proposition \ref{thm:c1}.

For each $\iii \in I^*$ and $t \ge 0$ the function $\hhh \mapsto
\bigl|\fii_\iii'\bigl( \pi(\hhh) \bigr)\bigr|^t$ defined on $I^\infty$
is a cylinder function satisfying the chain rule, see \cite[Chapter 2]{AK2},
and hence, by the open set condition, \eqref{eq:bc}, and
\cite[Theorems 2.5, 3.7, and 3.8]{AK2}, there exists a 
Borel probability measure $m$ on $E$ such that for each $\iii \in I^*$
\begin{equation} \label{eq:3.x}
  m\bigl( \fii_\iii(E) \bigr) = \int_E |\fii_\iii'(x)|^t dm(x),
\end{equation}
where $t = \dimh(E)$. The measure $m$ is called a \emph{conformal measure}.
See also \cite{JH}, \cite{MU}, and \cite{AK2}.
It can be easily shown that there exists a constant $C>0$ such
that
\begin{equation} \label{eq:puoli_ahlfors}
  m\bigl( B(x,r) \bigr) \ge Cr^t
\end{equation}
for all $x \in E$ and $0<r<r_0$. Namely, take $\iii = (i_1,i_2,\ldots)
\in I^\infty$ such that $\pi(\iii) = x$ and $n$ to be the smallest
integer for which $\fii_{\iii|_n}(E) \subset B(x,r)$. Now, using \ref{F3},
\eqref{eq:BDP}, and Lemma \ref{thm:bdpomin}(3), we obtain
\begin{align*}
  m\bigl( B(x,r) \bigr) &\ge m\bigl( \fii_{\iii|_n}(E) \bigr)
  = \int_E |\fii_{\iii|_n}'(x)|^t dm(x) \\
  &= \int_E \bigl| \fii_{\iii|_{n-1}}'\bigl( \fii_{i_n}(x) \bigr) \bigr|^t
  |\fii_{i_n}'(x)|^t dm(x) \\
  &\ge K_0^{-2t} \min_{i \in I}||\fii_i'||^{t} ||\fii_{\iii|_{n-1}}'||^t \\
  &\ge D^{-t} K_0^{-2t} \min_{i \in I}||\fii_i'||^{t}
  \diam\bigl( \fii_{\iii|_{n-1}}(X) \bigr)^t,
\end{align*}
where $t = \dimh(E)$. The claim follows since the set $\fii_{\iii|_{n-1}}(X)$
is not included in $B(x,r)$. For the inequality to the other
direction, the reader is referred to \cite[proof of Theorem 3.8]{AK2}.

We are now ready to prove the main theorem.

\begin{theorem} \label{thm:tangential}
  Suppose $t = \dimh(E)$ and $0 < l < d$. If a point
  of $E$ has a weak $(t,l)$-tangent plane, then $E$ is uniformly
  $l$-tangential.
\end{theorem}

\begin{proof}
  Let us first sketch the main idea of the proof: Assuming that the
  conclusion fails, so that there exists a
  point $x \in E$ with no tangent, we find for each plane $W$ a point
  $y \in E$ close to $x$ such that the angle between $y-x$ and $W$ is
  large. Since the set $\{ \fii_\iii(x) : \iii \in I^* \}$ is dense in 
  $E$, we are able to, using Lemma \ref{thm:smallangles}, map this
  setting arbitrary close to any given point in $E$. Hence, if $a \in
  E$ has a weak tangent plane $V$, we obtain an immediate contradiction,
  since either the image of $x$ or the image of $y$ is not included in
  a small neighborhood of $V + \{ a \}$ provided that $W$ is chosen
  in the beginning such that the image of $W$ is close to $V$.

  Suppose $a \in E$ has a weak $(t,l)$-tangent plane $V$. Assume on
  the contrary that there is $0<\delta<1$ such that for each $q \in
  \N$ there exists $x_q \in E$ such that for every $W \in G(d,l)$
  \begin{equation} \label{eq:tang_eka}
    E \cap B(x_q,1/q) \setminus X(x_q,W,\delta) \ne \emptyset.
  \end{equation}
  Put $1/(\delta+1) < \roo < 1$ and let $r_0 = r_0(1/\roo-\delta,\roo)
  < \dist(E,\partial\Omega)$ be as in Lemma \ref{thm:smallangles}. Fix
  $q \in \N$ such that $1/q < r_0/2$ and, to simplify the notation,
  denote $x_q$ with $x$. Take $\iii \in I^\infty$ such that
  $\pi(\iii)=a$. Then clearly $\fii_{\iii|_k}(x) \to a$ as $k \to
  \infty$. Setting $A_k = \fii_{\iii|_k}'(x)/|\fii_{\iii|_k}'(x)|$ for all $k
  \in \N$ and using the compactness of $G(d,l)$, we notice
  $\{ A_k^{-1}V \}_{k \in \N}$ has a subsequence converging to some $W \in
  G(d,l)$. Denoting the subsequence as the original sequence and
  setting $W_k = A_kW$, we have $W_k \to V$ as $k \to \infty$.
  Choosing $y \in E \cap B(x,1/q) \setminus X(x,W,\delta)$, we notice
  there exists $0<\eta<1$ depending only on $\delta$ and $\roo$
  such that
  \begin{equation} \label{eq:tang_toka}
    B(y,\eta r') \subset B(x,r') \setminus X(x,W,\roo\delta),
  \end{equation}
  where $r' = 2|x-y|$. Applying Lemma
  \ref{thm:smallangles} we obtain
  \begin{equation} \label{eq:tang_kolmas}
  \begin{split}
    \fii_{\iii|_k}\bigl( B(x,r') &\setminus X(x,W,\roo\delta) \bigr) =
    \fii_{\iii|_k}\bigl( \overline{X(x,r',W^\perp,1-\roo\delta)} \bigr) \\
    &\subset \overline{X\bigl( \fii_{\iii|_k}(x), ||\fii_{\iii|_k}'||r',
    W_k^\perp, (1 - \roo\delta)/\roo \bigr)} \\
    &= B\bigl( \fii_{\iii|_k}(x), ||\fii_{\iii|_k}'||r' \bigr) \setminus
    X\bigl( \fii_{\iii|_k}(x), W_k, \delta-(1/\roo-1) \bigr)
  \end{split}
  \end{equation}
  whenever $k \in \N$.
  Hence, using Lemma \ref{thm:bdpomin}(1),
  \eqref{eq:tang_toka}, and \eqref{eq:tang_kolmas}, we have
  \begin{equation} \label{eq:tang_1}
  \begin{split}
    B\bigl( \fii_{\iii|_k}(y), &\,K_0^{-1}|\fii_{\iii|_k}'(y)|\eta r' \bigr)
    \subset \fii_{\iii|_k}\bigl( B(y,\eta r') \bigr) \\
    &\subset B\bigl( \fii_{\iii|_k}(x),||\fii_{\iii|_k}'||r' \bigr) \setminus
    X\bigl( \fii_{\iii|_k}(x), W_k, \delta-(1/\roo-1) \bigr)
  \end{split}
  \end{equation}
  whenever $k \in \N$.
  Since $W_k \to V$ as $k \to \infty$, we may take $k_0$ large
  enough such that $|Q_{W_k} - Q_V| <
  2^{-1}(\delta-(1/\roo-1))^{1/2}$ whenever $k\ge k_0$.
  Recalling that the set $X(0,V,\delta)$ is an open ball in $G(d,l)$
  centered at $V$ with radius $\delta^{1/2}$, we notice, using the
  triangle inequality, that
  \begin{equation} \label{eq:tang_3}
    X\bigl( \fii_{\iii|_k}(x), V, (\delta-(1/\roo-1))/4 \bigr) \subset
    X\bigl( \fii_{\iii|_k}(x), W_k, \delta-(1/\roo-1) \bigr) 
  \end{equation}
  whenever $k \ge k_0$.

  Let $r>0$ and choose $n$ to be the smallest integer for which
  \begin{equation*}
    ||\fii_{\iii|_n}'|| < D^{-1}r/2.
  \end{equation*}
  By choosing $r>0$ small enough, we may assume that $n \ge k_0$.
  Since by \eqref{eq:tang_1} and \eqref{eq:tang_3}
  \begin{equation*}
  \begin{split}
    B\bigl( \fii_{\iii|_n}(y),K_0^{-1}|\fii_{\iii|_n}'(y)|\eta r' \bigr)
    \subset\; &B\bigl( \fii_{\iii|_n}(x),||\fii_{\iii|_n}'||r' \bigr)
    \setminus \\ &X\bigl( \fii_{\iii|_n}(x), V, (\delta-(1/\roo-1))/4
    \bigr),
  \end{split}
  \end{equation*}
  this choice gives, using \ref{F3} and \eqref{eq:BDP},
  \begin{equation*}
  \begin{split}
    \bigl| Q_V\bigl( \fii_{\iii|_n}&(x) - \fii_{\iii|_n}(y) \bigr)
    \bigr| \ge 2^{-1}(\delta-(1/\roo-1))^{1/2} |\fii_{\iii|_n}(x) -
    \fii_{\iii|_n}(y)| \\
    &\ge 2^{-1}(\delta-(1/\roo-1))^{1/2} K_0^{-1}|\fii_{\iii|_n}'(y)|
    \eta r' \\
    &\ge 2^{-1}(\delta-(1/\roo-1))^{1/2} K_0^{-2}\eta r'
    ||\fii_{\iii|_{n-1}}'|| \min_{i \in I}|\fii_i'(y)| \\
    &\ge 2^{-1}(\delta-(1/\roo-1))^{1/2} K_0^{-2}\eta r' \min_{i \in
    I}|\fii_i'(y)| D^{-1}r/2 \\
    & =: \lambda r,
  \end{split}
  \end{equation*}
  where $\lambda > 0$ does not depend on $r$. Assuming now
  $\dist\bigl( \fii_{\iii|_n}(x) - a, V \bigr) \le \lambda r/2$, we
  have
  \begin{equation*}
  \begin{split}
    \dist\bigl( \fii_{\iii|_n}(y) - a, V \bigr) &\ge \bigl| Q_V\bigl(
    \fii_{\iii|_n}(x) - \fii_{\iii|_n}(y) \bigr) \bigr| - \bigl| Q_V\bigl(
    \fii_{\iii|_n}(x) - a \bigr) \bigr| \\
    & \ge \lambda r - \lambda r/2 = \lambda r/2.
  \end{split}
  \end{equation*}
  Changing the roles of $x$ and $y$ above, we observe that there
  exists $z \in \{ x,y \}$ such that
  \begin{equation*}
    \dist\bigl( \fii_{\iii|_n}(z) - a, V \bigr) \ge \lambda r/2.
  \end{equation*}
  Since by Lemma \ref{thm:bdpomin}(3)
  \begin{equation*}
  \begin{split}
    \dist\bigl( \fii_{\iii|_n}(z) - a, V \bigr) &\le
    |\fii_{\iii|_n}(z) - a|
    \le \diam\bigl( \fii_{\iii|_n}(X) \bigr) \\
    &\le D||\fii_{\iii|_n}'|| < r/2,
  \end{split}
  \end{equation*}
  we have
  \begin{equation*}
    B\bigl( \fii_{\iii|_n}(z),\lambda r/8 \bigr) \subset B(a,r)
    \setminus V_a(\lambda r/8).
  \end{equation*}
  Therefore, using \eqref{eq:puoli_ahlfors},
  \begin{equation*}
    m\bigl( B(a,r) \setminus V_a(\lambda r/8) \bigr) \ge
    C (\lambda/8)^t r^t
  \end{equation*}
  for all $r>0$. This contradicts the assumption that $V$ is a weak
  $(t,l)$-tangent plane of $E$ at $a$.
\end{proof}

Let us next discuss applications of this theorem. At first,
we study uniformly $l$-tangential sets of $\R^d$.
Our aim is to embed each such a set into a $C^1$-submanifold.

\begin{proposition} \label{thm:c1}
  If $0 < l < d$ and a closed set $A \subset \R^d$ is
  uniformly $l$-tangential, then $A$ is a 
  subset of an $l$-dimensional $C^1$-submanifold.
\end{proposition}

\begin{proof}
  Take $a \in A$ and denote the $l$-tangent plane associated to a
  point $x \in A$ with $V_x$. We shall prove that there exists $r_0>0$
  not depending on $a$ such that
  \begin{equation} \label{eq:c1_goal}
    A \cap B(a,r_0) \subset X(x,V_a,1/2)
  \end{equation}
  whenever $x \in A \cap B(a,r_0)$. From this the claim
  follows by applying Whitney's Extension Theorem to the bi-Lipschitz
  mapping $P_{V_a}^{-1} \colon P_{V_a}\bigl( A \cap
  \overline{B}(a,r_0) \bigr) \to V_a^\bot$ (we identify $\R^d$ with
  the direct sum $V_a + V_a^\bot$).
  To prove \eqref{eq:c1_goal}, we shall first show that there exists
  $r_1 > 0$ such that
  \begin{equation} \label{eq:c1_1}
    d(V_x,V_a) < 1/8^{1/2}
  \end{equation}
  for every $x \in A \cap B(a,r_1)$. Suppose this is not true. Then
  with any choice of $r>0$ there is $x \in A \cap B(a,r)$ for which
  $d(V_x,V_a) \ge 1/8^{1/2}$. Recalling that the set $X(0,V,\delta)$
  is an open ball in $G(d,l)$ centered at $V$ with radius
  $\delta^{1/2}$, we infer
  \begin{equation*}
    X(0,V_x,1/32) \cap X(0,V_a,1/32) = \emptyset.
  \end{equation*}
  Hence $x \notin X(a,V_a,1/32)$ or $a \notin
  X(x,V_x,1/32)$. According to the assumptions, both cases are clearly
  impossible provided that $r>0$ is chosen small enough.

  Observe that \eqref{eq:c1_1} implies immediately that
  \begin{equation*}
    X(x,V_x,1/8) \subset X(x,V_a,1/2)
  \end{equation*}
  whenever $x \in A \cap B(a,r_1)$. Using the assumptions, we choose
  $r_2>0$ such that
  \begin{equation*}
    A \cap B(x,r_2) \subset X(x,V_x,1/8).
  \end{equation*}
  Now, defining $r_0 = \min\{ r_1,r_2/2 \}$, we have shown
  \eqref{eq:c1_goal} and therefore finished the proof.

\end{proof}

The generalizations for Theorems \ref{thm:mayerurbanski} and
\ref{thm:kaenmaki} are now straightforward.

\begin{corollary}
  Suppose $l = \dimt(E)$. Then either

  (1) $\dimh(E) > l$ or

  (2) $E$ is contained in an $l$-dimensional $C^1$-submanifold.
\end{corollary}

\begin{proof}
  The claim follows from \cite[Lemma 2.1]{VMU}, Theorem
  \ref{thm:tangential}, Proposition \ref{thm:c1}, and the fact that
  $\HH^t(E) > 0$ as $t = \dimh(E)$ (see \cite[Theorem 3.8]{AK2}).
  Observe that in \cite[Lemma 2.1]{VMU} one does not need the mappings
  $\fii_i$ to be conformal. The existence of the conformal measure
  will suffice.
\end{proof}

\begin{corollary}
  Suppose $t = \dimh(E)$ and $0<l<d$. Then either

  (1) $\HH^t(E \cap M)=0$ for every $l$-dimensional
  $C^1$-submanifold $M \subset \R^d$ or

  (2) $E$ is contained in an $l$-dimensional $C^1$-submanifold.
\end{corollary}

\begin{proof}
  The claim follows from \cite[Lemma 2.2]{AK1}, Theorem
  \ref{thm:tangential}, and Proposition \ref{thm:c1}.
  Observe that in \cite[Lemma 2.2]{AK1} one does not need the mappings
  $\fii_i$ to be conformal. The existence of the conformal measure
  will suffice.
\end{proof}

\begin{ack}
  The author expresses his gratitude to Professor Pertti Mattila for
  his valuable comments on the manuscript.
\end{ack}

\bibliographystyle{abbrv}
\bibliography{rigidity.bib}

\end{document}